\documentclass[a4paper,11pt]{amsart}
\usepackage[english]{babel}
\usepackage{amssymb,amsmath,amsthm}
\usepackage{verbatim}

\newtheorem{theorem}{Theorem}[section]
\newtheorem{corollary}[theorem]{Corollary}
\newtheorem{conjecture}[theorem]{Conjecture}
\newtheorem{lemma}[theorem]{Lemma}
\numberwithin{equation}{section}

\def\be{\begin{equation}}
\def\ee{\end{equation}}
\def\bal{\begin{aligned}}
\def\eal{\end{aligned}}
\def\bcs{\begin{cases}}
\def\ecs{\end{cases}}
\def\lb{\label}
\def\={\;=\;}
\def\+{\,+\,}
\def\-{\,-\,}

\def\to{\rightarrow}
\def\Co{\mathbb C}
\def\R{\mathbb R}
\def\Q{\mathbb Q}
\def\N{\mathbb N}
\def\Z{\mathbb Z}
\def\e{\varepsilon}
\def\diag{\rm diag}

\title[Nahm's conjecture]{Nahm's conjecture: asymptotic computations and counterexamples}
\author{Masha Vlasenko and Sander Zwegers}
\date{\today}
\address{Max-Planck-Institut f\"ur Mathematik, Vivatsgasse 7, 53111 Bonn, Germany}
\address{Mathematisches Institut, Universit\"at zu K\"oln, Weyertal 86--90, 50931 K\"oln, Germany}
\email{vlasenko@mpim-bonn.mpg.de}
\email{szwegers@uni-koeln.de}

\begin{document}
\maketitle
\section{Introduction}
Let $r\ge1$ be a positive integer, $A$ a real positive definite symmetric $r \times r$-matrix, $B$ a vector of length $r$, and $C$ a scalar. The series
\be\lb{F}
F_{A,B,C}(q) \= \sum_{n \in (\Z_{\ge0})^r} \frac{q^{\frac 12 n^T An + n^T B + C} }{(q)_{n_1}\dots (q)_{n_r}}.
\ee
converges for $|q|< 1$. Here we use the notation $(a;q)_n:=\prod_{k=1}^n(1-aq^{k-1})$ for $n\in\Z_{\geq 0} \cup \{\infty\}$ and the convention that the second argument is removed if it equals $q$ (so $(q)_n=(q;q)_n=\prod_{k=1}^n(1-q^k)$). We are concerned with the following problem due to Werner Nahm \cite{Nahm1,Nahm2,Nahm3}: describe all such $A, B$ and $C$ with rational entries for which~\eqref{F} is a modular form. In~\cite{Terhoeven1, Terhoeven2, Zagier} it was solved by Michael Terhoeven and Don Zagier for $r=1$ and the list contains seven triples $(A,B,C)\in\Q_+\times\Q\times\Q$. We develop this approach for $r>1$ and find several new examples of modular functions~\eqref{F} already for $r=2$.

Nahm has also given a conjectural criterion for a matrix $A$ to be such that there exist some $B$ and $C$ with modular $F_{A,B,C}$ (see~\cite{Nahm3}). The condition for the matrix $A$ is given in terms of solutions of a system of algebraic equations
\be\lb{ne}
1\-Q_i \= \prod_{j=1}^r Q_j^{A_{ij}}\,,\qquad i\=1,\dots,r\,.
\ee
In the last section we give several examples where the matrix $A$ doesn't satisfy the condition but corresponding modular forms exist. Certainly, it doesn't mean that the conjecture is completely wrong, rather that its correct formulation is an interesting open question.

\section{Asymptotical computations}
Let us explain a method to compute the asymptotics of~\eqref{F} when $q \to 1$. The idea comes from~\cite{Zagier}, where it is written in a very sketchy form. We denote the general term of the sum~\eqref{F} by $a_n(q)$. Suppose $q \to 1$ and $n_i \to \infty$ so that $q^{n_i} \to Q_i$ for some numbers $Q_i \notin \{0,1\}$. Then we have
\[
\frac{a_{n+e_i}}{a_{n}}\=\frac{q^{n^T A e_i + \frac12 e_i^T A e_i + e_i^T B}}{1-q^{n_i+1}} \to \frac{Q_1^{A_{i1}}\dots Q_r^{A_{ir}}}{1\-Q_i}\,,
\]
where $e_i$ is a vector whose all but $i$th coordinates are $0$ and $i$th coordinate is $1$.  We have the following statement.

\begin{lemma}\lb{Q} Let $A$ be a real positive definite symmetric $r\times r$ matrix. Then the system of equations
\[
1\-Q_i \= \prod_{j=1}^r Q_j^{A_{ij}}\,,\qquad i\=1,\dots,r
\]
has a unique solution with $Q_i\in(0,1)$ for all $1\leq i \leq r$.
\end{lemma}
\begin{proof}
We consider the function $f_A : [0,\infty)^r \rightarrow  \R$ given by
\[ f_A(x) =\frac 12 x^T A x +\sum_{i=1}^r Li_2(\exp(-x_i)),\]
where $Li_2$ is the dilogarithm function defined by the power series $Li_2(z)=\sum_{n=1}^\infty \frac{z^n}{n^2}$ for $|z|<1$. It has the property $z Li_2'(z)=-\log(1-z)$. 

The gradient and the Hessian of $f_A$ are
\[ \begin{aligned}
\nabla f_A(x) &= A x+ \left( \log(1-\exp(-x_i))\right)_{1\leq i\leq r},\\
H_{f_A} (x) &= A + \operatorname{diag} \left( \frac{1}{\exp(x_i)-1}\right)_{1\leq i\leq r}.
\end{aligned}\]

Using $Q_i = \exp(-x_i)$, the statement of the lemma is equivalent to saying that $f_A$ has a unique critical point in $(0,\infty)^r$.

First, $f_A$ has at least one critical point in $(0,\infty)^r$, because it takes on it's minimum in $(0,\infty)^r$: it's continuous, bounded from below by 0 and $f_A(x) \rightarrow \infty$ if $||x|| \rightarrow \infty$, and so it takes on it's minimum in $[0,\infty)^r$. In fact, it takes on that minimum in $(0,\infty)^r$, because 
\[ \lim_{x_i \downarrow 0} \frac{\partial f_A}{\partial x_i} (x) =-\infty<0.\]

Second, $f_A$ has at most one critical point in $(0,\infty)^r$, because it's differentiable and strictly convex on $(0,\infty)^r$: since $A$ is positive definite, we see that the Hessian $H_{f_A}(x)$ is positive definite for all $x\in (0,\infty)^r$.
\end{proof}

Consider the unique solution $Q_i\in(0,1)$ of~\eqref{ne} and let $q=e^{-\e}$, $\e>0$. Then one has
\[
\frac{a_{n+e_i}(q)}{a_{n}(q)}\approx 1 \quad \forall i \quad \text{ when } \quad n \approx \bigl(-\frac{\log Q_1}{\e},\dots,-\frac{\log Q_r}{\e}\bigr)\,,
\]
and it is very likely that $a_{n}(q)$ as a function of $n$ is maximal around this point. We will apply a version of Laplace's method to describe the asymptotics of $F_{A,B,C}(e^{-\e})$ for small $\e$. For this we need the so called polylogarithm 
\[ Li_m (z) = \sum_{n=1}^\infty \frac{z^n}{n^m} \qquad \text{for}\quad |z|<1,\ m\in\Z, \]
which satisfies the obvious relation
\[ z\frac{d}{dz} Li_m(z)  = Li_{m-1} (z).\]

\begin{lemma}\lb{qn} Let $n \in \N$ and $q=e^{-\e}$ with $\e>0$.  We fix $Q \in (0,1)$ and introduce a variable $\nu = -\log Q -n\e$. Then  

(i) for all $n,\e$ we have an inequality
\be\lb{qn1}
\log\Bigl( \frac{(q)_{\infty}}{(q)_n} \Bigr)\;<\; \-\frac{Li_2(Q)}{\e}\+\Bigl(\frac{\nu}{\e}-\frac12\Bigr)\log(1-Q)\+\frac{\nu}2\frac{Q}{1-Q}\,;
\ee
 
(ii) we have an asymptotic expansion
\be\lb{qn2}
\log\Bigl( \frac{(q)_{\infty}}{(q)_n} \Bigr) \;\sim\; -\,
\sum_{r,s\ge0} \frac{Li_{2-r-s}(Q) B_r}{r!s!} \nu^s \e^{r-1}\;\text{ when }\;\e,\nu \to 0\,,
\ee
where $B_0=1, B_1=-\frac12, B_2=\frac 16,\dots$ are the Bernoulli numbers.
\end{lemma}
\begin{proof}
\[\bal
\log \Bigl(\frac{(q)_{\infty}}{(q)_n}\Bigr) &\= \sum_{s=1}^{\infty} \log\bigl(1-q^{n+s} \bigr) \= \sum_{s=1}^{\infty} \log\bigl(1- Q e^{\nu- s\e} \bigr)\\
&\= -\sum_{s=1}^{\infty} \sum_{p=1}^{\infty} \frac{Q^p e^{p(\nu-s\e)}}{p} \= - \sum_{p=1}^{\infty} \frac{Q^p}{p}\frac{e^{p\nu}}{e^{p\e}-1}
\eal
\]
Since $e^x>1+x$ for all $x\ne0$ and $\frac{x}{e^x-1}>1+\frac x2$ for $x>0$ then
\[\frac{e^{p\nu}}{e^{p\e}-1} \;>\; (1+p\nu)\Bigl(\frac1{p\e}\-\frac12\Bigr) \= \frac1{p\e}+\Bigl(\frac{\nu}{\e}-\frac12\Bigr)-p\frac{\nu}{2}\,,\]
and we get inequality~(i) after summation in $p$. To prove~(ii) we notice that for every fixed $p$ we have an asymptotic expansion
\[
\frac{p\e e^{p\nu}}{e^{p\e}-1} \;\sim\; \Bigl( \sum_{r=0}^{\infty} \frac{B_r}{r!}(p\e)^r\Bigr)\Bigl(\sum_{s=0}^{\infty} \frac{(p\nu)^s}{s!}\Bigr) \=  \sum_{r,s\ge 0} \frac{B_r}{r! s!}(p \e)^r (p\nu)^s\,,
\]
i.e.\ for every fixed $N$ and $\delta>0$ we can find $\delta' > 0$ such that 
\[\frac{\Bigl|\frac{p\e e^{p\nu}}{e^{p\e}-1}-\sum_{r+s\le N} \frac{B_r p^{r+s}}{r! s!}\e^r \nu^s\Bigr|}{p^N \max(\e,|\nu|)^N} < \delta\]
whenever $p\e, p|\nu| < \delta'$. Also we observe that when $x \searrow 0$ 
\be\lb{lim1}
\frac1{x^N} \sum_{p>\frac{\delta'}{x}} p^a Q^p \to 0
\ee
for any $a$, as well as
\be\lb{lim2}
\frac1{x^N} \sum_{p>\frac{\delta'}{x}} \frac{Q^p}{p^2}\frac{p\e e^{p\nu}}{e^{p\e}-1} \;<\;
\frac1{x^N} \sum_{p>\frac{\delta'}{x}} Q^p e^{p\nu}  \;<\; \frac1{x^N}\frac{e^{\frac{\delta'}x(\nu+\log Q)}}{1-e^{\nu+\log Q}}
\to 0
\ee
uniformly in $\nu$ in small domains. Let us choose $\delta''>0$ such that expressions~\eqref{lim1} for all integer $a$ between $-2$ and $N-2$ and also the left-hand side of~\eqref {lim2} are smaller than $\delta$ whenever $x<\delta''$ and $|\nu|<\delta''$. Now if $\max(\e,|\nu|)<\delta''$ then
\[\bal
&\frac{\Bigl|\sum_{p\ge 1} \frac{Q^p}{p^2}\Bigl(\frac{p\e e^{p\nu}}{e^{p\e}-1}-\sum_{r+s\le N} \frac{B_r p^{r+s}}{r! s!}\e^r \nu^s\Bigr)\Bigr|}{\max(\e,|\nu|)^N}
\le \delta \sum_{p \max(\e,|\nu|) < \delta'} p^{N-2}Q^p \\
&\qquad\+ \frac1{\max(\e,|\nu|)^N}\sum_{p \max(\e,|\nu|) > \delta'}  \frac{Q^p}{p^2} \frac{p\e e^{p\nu}}{e^{p\e}-1}\\
&\qquad\+ \sum_{r+s\le N} \frac{|B_r|}{r!s!}\e^r|\nu|^s \, \frac1{\max(\e,|\nu|)^N} \sum_{p \max(\e,|\nu|) > \delta'} p^{r+s} Q^p\\
&\qquad \le \Bigl(L_{2-N}(Q)+1+\sum_{r+s\le N} \frac{|B_r|}{r!s!}(\delta'')^{r+s}\Bigr)\delta
\eal \]
and~(ii) follows.
\end{proof}

\vskip 2pc

Let $B_p(X)=\sum_k \binom pk B_k X^{p-k}$, $p \ge 1$ be the Bernoulli polynomials. Consider polynomials $D_p \in \Q[B,X,T]$, $p \ge 1$ defined by the following equality of formal power series in $\e^{1/2}$:
\be\lb{Dp}\bal
\exp\Bigl[(B+\frac12\frac{Q}{1-Q})T\e^{1/2} \- \sum_{p=3}^{\infty}&\frac1{p!} B_p\bigl(\frac{T}{\e^{1/2}}\bigr) Li_{2-p}(Q)\,\e^{p-1}\Bigr]\\
& \= 1+\sum_{p=1}^{\infty} D_p\bigl(B,\frac{Q}{1-Q},T\bigr)\,\e^{p/2}\,.
\eal\ee
Observe that the coefficients of the series under the exponent are polynomials in $B,\frac{Q}{1-Q}$ and $T$ because $Li_{2-r}(Q)=P_{r-1}\Bigl(\dfrac Q{1-Q}\Bigr)$ where $P_r$, $r \ge 1$ are the polynomials defined by $P_1(X)=X$ and $P_{p+1}(X)=(X^2+X)\frac d{dX}P_{p}(X)$.

\begin{theorem}\lb{asthm}
There is an asymptotic expansion
\[
F_{A,B,C}(e^{-\e})\, e^{-\frac{\alpha}{\e}} \;\sim\; \beta e^{-\gamma \e}\Bigl(1+\sum_{p=1}^{\infty} c_p \e^p \Bigr)\,,\quad \e \searrow 0
\]
with the coefficients $\alpha \in \R_+$, $\beta,\gamma \in \overline{\Q}$ and $c_p \in \overline{\Q}$, $p \ge 1$ given below. Let $Q_i\in(0,1)$ be the solutions of~\eqref{ne}. Denote $\xi_i=\frac{Q_i}{1-Q_i}$, $\widetilde A \= A + \diag \{\xi_1,\dots,\xi_r\}$ and let $L(x)$ be the Rogers dilogarithm function. Then
\[\bal
&\alpha\=\sum_{i=1}^r\bigl(L(1)-L(Q_i)\bigr) \,>\, 0\,,\\
&\beta\=\det \widetilde A^{-1/2} \prod_i Q_i^{B_i} (1-Q_i)^{-1/2} \,,\quad \gamma\=C+\frac 1{24}\sum\frac{1+Q_i}{1-Q_i}\,,\\
&c_p \= \det \widetilde A^{1/2}(2\pi)^{-r/2} \int C_{2p}(B, \xi, t) e^{-\frac 12 t^T \widetilde A t} dt\,,\\
\eal\]
where the polynomials in $3r$ variables $C_p \in \Q[B,\xi,t]$ are defined as 
\[
 C_p(B,\xi,t) = \sum_{p_1+\dots+p_r=p} \prod_{i=1}^p D_{p_i}(B_i,\xi_i,t_i)\,,
\]
where $D_p$ are the polynomials in 3 variables defined by~\eqref{Dp}.
\end{theorem}

Recall that $L(x)$ is an increasing function on $\R$ (therefore $\alpha > 0$), we have $L(x)=Li_2(x)+\frac12\log(x)\log(1-x)$ for $x \in (0,1)$  and $L(1)=\frac{\pi^2}6$.

\begin{proof} 
Let 
\[
 \alpha'\=-\sum_{i=1}^r L(Q_i)\,,\quad \beta'\= \prod_i Q_i^{B_i} (1-Q_i)^{-1/2}\,,\quad \gamma'\=C+\frac1{12}\sum_i\xi_i
\]
and $t_i \= -\frac{\log Q_i}{\e}-n_i$. Consider the function
\[
\phi(t,\e) \= \frac{(q)_{\infty}^r a_{n}(q)}{\beta' e^{\frac{\alpha'}{\e}}}\qquad(q=e^{-\e})
\]
defined only for $t \in t^0(\e)+\Z^r$ where $t^0_i(\e)$ is the fractional part of $-\frac{\log Q_i}{\e}$.  After a straightforward computation using~(i) of Lemma~\ref{qn} we obtain that
\be\lb{an1'}
\log \phi(t,\e) \,<\, \Bigl(- \frac 12 t^T A t + t^T \bigl(B+\frac12 \xi\bigr) \- C \Bigr) \e\,.
\ee
Then
\[
\frac{(q)_{\infty}^r F_{A,B,C}(q)}{\beta' \exp(\frac{\alpha'}{\e})} \= \sum_{t \in t^0 + \Z^r} \phi(t,\e) \quad \sim \sum_{t \in t^0 + \Z^r, |t_i|<{\e}^{\lambda}} \phi(t,\e) 
\]
for every $\lambda < -\frac12$, where "$\sim$" always means that the difference is $o(\e^N)$ for every $N$. Indeed, for such $\lambda$ we have $\sum_{|t_i|>{\e}^{\lambda}} \phi(t,\e) \= o(\e^N)$ for every $N$ due to~\eqref{an1'}.  We can further rewrite it as
\[\sum_{t \in t^0 + \Z^r, |t_i|<{\e}^{\lambda}} \phi(t,\e) \quad \= \sum_{t \in (t^0 + \Z^r)\sqrt{\e}, |t_i|<{\e}^{\lambda+\frac12}} \phi\bigl(\frac{t}{\sqrt{\e}},\e\bigr)\,.
\]
Let also $\lambda > -\frac23$. Then
\be\lb{ass}
\phi\bigl(\frac{t}{\sqrt{\e}},\e\bigr) \= e^{- \frac 12 t^t \widetilde A t - \gamma' \e}\Bigl(1+\sum_{p=1}^N C_p(t) \e^{p/2}\Bigr) \+ o(\e^{N(3\lambda+2)})
\ee
uniformly in the domain $|t_i|\le\e^{\lambda+\frac12}$, and we observe that  for any polynomial $P$
\be\lb{Four}
 \sum_{t \in (t^0 + \Z^r)\sqrt{\e}, |t_i|<{\e}^{\lambda+\frac12}} P(t)  e^{- \frac 12 t^T \widetilde A t} \;\sim\; \e^{-r/2} \int P(t)  e^{- \frac 12 t^T \widetilde A t} dt
\ee
(the difference is $o(\e^N)$ for every $N$) when $\lambda<-\frac12$. Combining~\eqref{ass} and~\eqref{Four} (we will prove both facts later), we get
\[
 \frac{(q)_{\infty}^r F_{A,B,C}(q)}{\beta' \exp(\frac{\alpha'}{\e}-\gamma' \e)} \;\sim\; \Bigl(\frac{2\pi}{\e}\Bigr)^{r/2} {\det\widetilde A}^{-1/2} \Bigl(1+\sum_{p=1}^{\infty} c_p \e^p \Bigr)\,.
 \]
 Here we have only integer powers of $\e$ because $\int C_p(t) e^{- \frac 12 t^T \widetilde A t} dt = 0$ when $p$ is odd. And this happens because the total $t$-degree of every monomial in $C_p$ has the same parity as $p$, which in turn follows from the definition of $D_p$. Now, since 
\[
 \log (q)_{\infty} \;\sim\; -\frac{\pi^2}{6}\frac1{\e} \+ \frac12\log\Bigl(\frac{2\pi}{\e}\Bigr) \+ \frac{\e}{24}
\]
when $\e\to 0$, we obtain the statement of the theorem.  

To prove~\eqref{Four} we notice again that $\sum_{|t_i|>{\e}^{\lambda+\frac12}} P(t) e^{- \frac 12 t^t \widetilde A t} \= o(\e^N)$ for every $N$, and using Poisson summation formula we have
\[
\sum_{t \in (t^0 + \Z^r)\sqrt{\e}} P(t)  e^{- \frac 12 t^T \widetilde A t} \=  \sum_{t \in t^0 + \Z^r} P(t\sqrt{\e})  e^{- \frac {\e}2 t^T \widetilde A t} \=
\sum_{s \in \Z^r} g(s) e^{2\pi i s^T t^0}\,,
\]
where $g(s)$ is the Fourier transform of $P(t\sqrt{\e}) e^{- \frac {\e}2 t^T \widetilde A t}$.  Then $g(0)$ is the right-hand side of~\eqref{Four}, and the sum of all remaining terms are $o(\e^N)$ since for any monomial $P'(t)$ and $g'(s)$ being the Fourier transform of $P'(t)e^{- \frac {\e}2 t^t A' t}$ one can check by direct computation that $\sum_{s \in \Z^2\backslash \{0\}} |g'(s)| = o(\e^N)$ for any $N$. 

It remains to prove~\eqref{ass}. Using~(ii) of Lemma~\ref{qn} we get
\[\bal
 \log \phi(t,\e) \;\sim\; & - \frac 12 t^T\widetilde A t - \gamma'\e + t^T \bigl(B+\frac12 \xi\bigr) \\
&\-  \sum_i \sum_{p=3}^{\infty} \frac1{p!} B_p(t_i) Li_{2-p}(Q_i) \e^{p-1}\,,\qquad \e,t\e \to 0\,,
\eal\]
and therefore for every $N$
\[\bal
\log \phi\bigl(\frac{t}{\sqrt{\e}},\e\bigr) \= - &\frac 12 t^T \widetilde A t - \gamma'\e  + t^T \bigl(B+\frac12 \xi\bigr)\sqrt{\e}  \\ & \-  \sum_i \sum_{p=3}^{N} \frac1{p!} B_p\Bigl(\frac{t_i}{\sqrt{\e}}\Bigr) Li_{2-p}(Q_i) \e^{p-1} \+ o(\e^{N(\lambda+1)-1})
\eal\]
uniformly in $|t_i|\le{\e}^{\lambda+\frac12}$. If we rewrite the right-hand side as $\sum_{p=0}^{N-2} g_p(t) \e^{\frac p2}$ then $\deg g_p \le p+2$ (because $\deg B_p = p$). It follows that $\sum_{p=1}^{N-2} g_p(t) \e^{\frac p2} = O({\e}^{3 \lambda \+ 2})$ uniformly in our domain since 
\[(\lambda + \frac 12)(p+2)\+\frac p2 \= p(\lambda+1)\+2\lambda \+ 1 \,\ge\, 3 \lambda \+ 2 \;>\; 0\,.\]
Therefore we can take a sufficiently long but finite part of the standard series to approximate its exponent. Hence some sufficiently long but again finite part of
\[
\exp \Bigl[\sum_i (B_i + \frac12\xi)t_i\sqrt{\e} \-  \sum_i \sum_{p=3}^{\infty} \frac1{p!} B_p\bigl(\frac{t_i}{\sqrt{\e}}\bigr) Li_{2-p}(Q_i) \e^{p-1} \Bigr] \= 1 \+ \sum_{p=1}^{\infty} C_p(t) \e^{p/2} 
\]
will approximate $\phi\bigl(\frac{t}{\sqrt{\e}},\e\bigr) e^{\frac 12 t^T \widetilde A t + \gamma' \e}$. One can easily see that $\deg C_p(t) \le 3p$ (in the variable $t$). Since for $p>N$ 
\[
C_p(t) \e^{\frac p2}\=O({\e}^{(\lambda+\frac12)3p+\frac p2})\=O({\e}^{p(3 \lambda + 2)})\=o({\e}^{N(3 \lambda + 2)})
\]
then it is sufficient to consider only the part with $p \le N$ in~\eqref{ass}.
\end{proof}

\section{Modular functions $F_{A,B,C}$}

Let us search for those triples $(A,B,C)$ for which $F_{A,B,C}(q)$ is a modular function (of any weight and any congruence subgroup).  We will call such $(A,B,C)$ \emph{a modular triple}. The idea here is that in order for $F_{A,B,C}(q)$ to be modular, the asymptotic expansion needs to be of a special type, as we can see from the following lemma.

\begin{lemma} Let $F(q)$ be a modular form of weight $w$ for some subgroup of finite index $\Gamma \subset {\rm SL}(2,\Z)$. Then for some numbers $a \in \pi^2 \Q$ and $b \in \Co$
\be\lb{modas}
e^{\frac a{\e}} F(e^{-\e}) \,\sim \, b \, \e^{-w} \+ o(\e^N)\qquad \forall N \ge 0\,.
\ee
\end{lemma}
\begin{proof}
The group ${\rm SL}(2,\Z)$ acts on the space $M_w(\Gamma)$ of modular functions of weight $w$. Let $S=\begin{pmatrix}0&1\\-1&0\end{pmatrix}$. Then $SF \in M_w(\Gamma)$, and in particular it has a $q$-expansion (with some rational powers of $q=e^{2\pi i z}$):
\[
 \frac1{z^w} F\Bigl(e^{-2\pi i\frac 1z}\Bigr)\= a_0 q^{\alpha_0} \+ a_1 q^{\alpha_1} \+ \dots
\]
We subsitute $z=\frac{2\pi i}{\e}$ and get
\[\bal
 F(e^{-\e}) &\=\Bigl(\frac{2\pi i}{\e}\Bigr)^w \Bigl[ a_0 e^{-\frac{4\pi^2\alpha_0}{\e}} \+ a_1 e^{-\frac{4\pi^2\alpha_1}{\e}}\+\dots\Bigr]\\
&\= \frac{(2\pi i)^w a_0}{\e^w} e^{-\frac{4\pi^2\alpha_0}{\e}} \Bigl[ 1 \+ o(\e^N)\Bigr] \qquad \forall N\,.
\eal\]

\end{proof}

If we now compare the asymptotics from Theorem~\ref{asthm} with~\eqref{modas} we get the following statement.

\begin{corollary}\label{modcor} If $F_{A,B,C}(q)$ is modular then
\begin{itemize}
\item[(i)] its weight $w=0$
\item[(ii)] $\alpha \in \pi^2 \Q \,\Longleftrightarrow\, \sum_{i=1}^r L(Q_i)\in\pi^2\Q$
\item[(iii)] $e^{-\gamma \e} \bigl(1+\sum_{p=1}^{\infty} c_p \e^p \bigr)=1 \,\Longleftrightarrow\, c_p=\frac{\gamma^p}{p!} \quad \forall p$
\end{itemize}
\end{corollary}

Condition~(ii) is very interesting, we consider it in the next section. It follows from~(iii) that modular triples satisfy an infinite number of equations 
\be\lb{eqns}
 \bigl(c_p\-\frac1{p!}c_1^p\bigr)\bigl(B,\xi,\widetilde{A}^{-1}\bigr) \= 0\,,\quad p=2,3,\dots\,,
\ee
and these equations are polynomial in the entries of $B,\xi,\widetilde{A}^{-1}$. Indeed, let us look at the expression for $c_p$ from Theorem~\ref{asthm}. Since the generating function for the moments of the Gaussian measure is
\[
\sum_{a \in (\Z_{\ge 0})^r} \frac{x^a}{a_1!\dots a_r!}  \frac{\det \widetilde A^{1/2}}{(2\pi)^{r/2}} \int t^a e^{-\frac 12 t^T \widetilde A t} dt \= \exp\bigl( \frac 12 x^T  \widetilde{A}^{-1} x \bigr)\,,
\]
all the moments are rational polynomials in the entries of $\widetilde{A}^{-1}$ and we obtain that $c_p\in\Q[B,\xi,\widetilde{A}^{-1}]$.

\vskip 2pc

Now let $r=1$. It is easy to see that the degrees of $D_p(B,X,T)$ in the variables $B, X$ and $T$ are $p,2p$ and $3p$, respectively. Since
$c_p(B,\xi,(A+\xi)^{-1})$ is the integral of $D_{2p}(B,\xi,t)$ w.r.t.\ the measure $\frac{(A+\xi)^{1/2}}{\sqrt{2\pi}}e^{-(A+\xi)t^2/2} dt$ and the integral of $t^{2m}$ is $(2m-1)!!(A+\xi)^{-m}$, the degrees of $c_p$ in the corresponding variables are $2p$, $4p$ and $3p$. It is convenient to consider the polynomials
\[
 \widetilde c_p (B,\xi,A)= (A+\xi)^{3p} \bigl[c_p-\frac 1{p!} c_1^p\bigr](B,\xi,\frac1{A+\xi})\,,\qquad p=2,3,\dots
\]
Although these polynomials look rather complicated, we have found using the \emph{Magma algebra system} (\cite{magma}) that the ideal 
\[I=\langle \widetilde c_2, \widetilde c_3, \widetilde c_4, \widetilde c_5 \rangle \subset \Q[B,\xi,A]\]
contains the element 
\[
\xi(\xi+1)A^{13}(A-1)^{13}(A+1)(A-2)(A-1/2)\,.
\]
Consequently, if $(A,B,C)$ is a modular triple then $A \in \{\frac12, 1, 2 \}$. For each $A$ on this list it is not hard to find the corresponding values of $B$, and one can compute $C$ from the equality $\gamma=c_1$. This way we obtain exactly the list from the theorem below.  

\begin{theorem}[D. Zagier \cite{Zagier}]\label{r1} Let $r=1$. The only $(A,B,C)\in \mathbb Q_{+} \times \mathbb Q \times \mathbb Q$ for which $F_{A,B,C}(q)$ is a modular form are given in the following table.
\end{theorem}
\begin{center}
\begin{tabular}{crr|c}
A&B&C&$F_{A,B,C}(e^{2\pi i z})$\\
\hline
&&&\\
2&0&$-1/60$&$\theta_{5,1}(z)/\eta(z)$\\
&&&\\
&1&$11/60$&$\theta_{5,2}(z)/\eta(z)$\\
\hline
&&&\\
1&0&$-1/48$&$\eta(z)^2/\eta(\frac z2)\eta(2z)$\\
&&&\\
 &$1/2$&$1/24$&$\eta(2z)/\eta(z)$\\
&&&\\
 &$-1/2$&$1/24$&$2\eta(2z)/\eta(z)$\\
\hline
&&&\\
1/2&0&$-1/40$&$\theta_{5,1}(\frac z4)\eta(2z)/\eta(z)\eta(4z)$\\
&&&\\
 &1/2&$1/40$&$\theta_{5,2}(\frac z4)\eta(2z)/\eta(z)\eta(4z)$\\
\end{tabular}
\end{center}
Here $\eta(z)=q^{1/24}\prod_{n=1}^{\infty}(1-q^n)$ and $\theta_{5,j}(z)=\sum_{n\equiv 2j-1(10)} (-1)^{[n/10]}q^{n^2/40}$.

We warn the reader that if~(iii) of Corollary~\ref{modcor} holds for some $(A,B,C)$ this does not yet imply that $F_{A,B,C}$ is in fact modular. To get modularity one needs to prove an identity between the corresponding $q$-series for each line of the table. For example, the first two lines correspond to the well known Rogers--Ramanujan identities.   

Further computer experiments showed that $\widetilde c_p \in I$ for $p=6,\dots,20$. Although we stopped at this point, it is very likely that the statement is true for all $p$. Also with the help of \emph{Magma} we have got the following decomposition of the radical of $I$ into prime ideals:
\[
 {\rm Rad}(I) \= \mathcal{P}_1 \cdot ... \cdot \mathcal{P}_{14} 
\]
where the generators of $\mathcal{P}_i$ are given below:
\begin{center}
\begin{tabular}{l|lll}
i & generators of $\mathcal{P}_i$ \\
\hline
1&&$\xi$&\\
2&&$\xi+1$&\\
3&$B-1/2,$&$\xi+2,$&$A$\\
4&$B-1,$&$ \xi+2,$&$ A$\\
5&$B,$&$ \xi+2,$&$ A$\\
6&$B+1/2,$&$ \xi^2+3\xi+1,$&$ A+1$\\
7&$B-1/2,$&$ \xi^2+3\xi+1,$&$ A+1$\\
8&$B+1/2,$&$ \xi-1,$&$ A-1$\\
9&$B,$&$ \xi-1,$&$ A-1$\\
10&$B-1/2,$&$ \xi-1,$&$ A-1$\\
11&$B-1,$&$ \xi^2-\xi-1,$&$ A-2$\\
12&$B,$&$ \xi^2-\xi-1,$&$ A-2$\\
13&$B-1/2,$&$ \xi^2+\xi-1,$&$ A-1/2$\\
14&$B, $&$\xi^2+\xi-1,$&$ A-1/2$\\
\end{tabular}
\end{center}
Consequently, the set of all solutions of the system $\widetilde c_p(B,\xi,A)=0$, $p=2,3,\dots$ is a subset of this table, and if we indeed had $\widetilde c_p \in I$ (or at least $\widetilde c_p \in {\rm Rad}(I)$) for all $p$ then this table would be exactly the set of solutions.

\vskip 2pc

Let's us now consider the case $r=2$. The task of solving the system~\eqref{eqns} for several small values of $p$ becomes already very complicated. We failed to solve it with \emph{Magma} in full generality for $r=2$ as we did in the case $r=1$. However, we can still search for modular $F_{A,B,C}$, where $A$ is of a special type. We will consider three families of matrices:
\be\bal\lb{fam}
 &A=\begin{pmatrix}a&\frac12-a\\\frac12-a&a\end{pmatrix}\quad \Rightarrow \quad \xi_1=\xi_2=\frac{\sqrt{5}-1}2 \\
 &A=\begin{pmatrix}a&2-a\\2-a&a\end{pmatrix}\quad \Rightarrow \quad \xi_1=\xi_2=\frac{\sqrt{5}+1}2 \\
 &A=\begin{pmatrix}a&1-a\\1-a&a\end{pmatrix}\quad \Rightarrow \quad \xi_1=\xi_2=\frac12\\ 
\eal\ee
It is easy to check that~(ii) of Corollary~\ref{modcor} holds for these matrices. For these families of matrices we can do an analysis similar to what we did for $r=1$.

\begin{theorem}\label{afam} Modular functions $F_{A,B,C}(z)$ with the matrix $A$ being of the form $\begin{pmatrix}a&\frac12-a\\\frac12-a&a\end{pmatrix}$ exist if and only if $a=1$, $a=3/4$ or $a=1/2$. Below is the list of all such modular functions.
\end{theorem}
\begin{tabular}{llr|c}
&&&\\
A&B&C&$F_{A,B,C}(e^{2\pi i z})$\\
\hline
&&&\\
$\begin{pmatrix}1&-\frac12\\-\frac12&1\end{pmatrix}$&$\begin{pmatrix}0\\0\end{pmatrix}$&$-\frac1{20}$& 
$\bal (\theta_{5,\frac 34} (2z) & +\theta_{5,\frac{13}{4}} (2z))\eta(z)/\eta(2z)\eta(z/2)\\&  +2\theta_{5,2} (2z)\eta(2z)/\eta(z)^2 \eal$\\
&$\begin{pmatrix}-\frac12\\0\end{pmatrix}$ and $\begin{pmatrix}0\\-\frac12\end{pmatrix}$&$\frac1{20}$&
$\bal &2 \theta_{5,1}(2z) \eta(2z)/\eta(z)^2 \\ +\theta_{5,\frac32} (z) \theta_{5,2}&(2z) \eta(z)^3/\eta(z/2)^2\eta(2z)^2\eta(10z)\eal$\\
\hline
&&&\\
$\begin{pmatrix}\frac34&-\frac14\\-\frac14&\frac34\end{pmatrix}$&$\begin{pmatrix}\frac14\\-\frac14\end{pmatrix}$ and $\begin{pmatrix}-\frac14\\\frac14\end{pmatrix}$&$-\frac1{80}$&$\theta_{5,1}(\frac z8)\eta(z)/\eta(\frac{z}2)\eta(2z)$\\
&&&\\
&$\begin{pmatrix}\frac12\\0\end{pmatrix}$ and $\begin{pmatrix}0\\\frac12\end{pmatrix}$&$\frac1{80}$&$\theta_{5,2}(\frac z8)\eta(z)/\eta(\frac{z}2)\eta(2z)$\\
\hline
&&&\\
$\begin{pmatrix}\frac12&0\\0&\frac12\end{pmatrix}$&$\begin{pmatrix}0\\0\end{pmatrix}$&$-\frac1{20}$&$(\theta_{5,1}(\frac z4)\eta(2z)/\eta(z)\eta(4z))^2$\\
&&&\\
&$\begin{pmatrix}\frac12\\0\end{pmatrix}$ and $\begin{pmatrix}0\\\frac12\end{pmatrix}$&$0$&$\theta_{5,1}(\frac z4)\theta_{5,2}(\frac z4)(\eta(2z)/\eta(z)\eta(4z))^2$\\
&&&\\
 &$\begin{pmatrix}\frac12\\\frac12\end{pmatrix}$&$\frac1{20}$&$(\theta_{5,2}(\frac z4)\eta(2z)/\eta(z)\eta(4z))^2$\\
\end{tabular}

\begin{proof}
Consider the ideal $I \subset \Q[b_1,b_2,\xi,a]$ generated by $\xi^2+\xi-1$ and the polynomials   
\[
(\xi^2+2a\xi+a-1/4)^{3p} \times \bigl[c_p-\frac{c_1^p}{p!} \bigr]\Bigl(\begin{pmatrix}b_1\\b_2\end{pmatrix},\begin{pmatrix}\xi\\\xi\end{pmatrix},\begin{pmatrix}a+\xi&\frac12-a\\\frac12-a&a+\xi\end{pmatrix}^{-1}\Bigr)
\]
for $p=2,3,4,5$. We find with \emph{Magma} that the element 
\[
 a(a-\frac14)(a-\frac12)(a-\frac34)(a-1)(a^2-a-\frac1{16})
\]
belongs to $I$. (We ran the function \emph{GroebnerBasis($I$)} which has computed the Groebner basis for $I$ using reversed lexicografical order on monomials with the variables ordered as $b_1>b_2>\xi>a$. It took several hours, the Groebner basis contains 15 elements, and the element above is one of them.) The last term doesn't give rational values for $a$, and the reason it enters here is that we have multiplied every equation $c_p-c_1^p/p!=0$ by $3p$th power of the determinant $\xi^2+2a\xi+a-1/4=(\xi+1/2)(\xi+2a-1/2)$ while precisely the denominator of $c_p-c_1^p/p!$ is $(\xi+1/2)^{3p}(\xi+2a-1/2)^{2p}$. Therefore our polynomials are divisible by $(\xi+2a-\frac12)^p$ for $p=2,3,4,5$, and since $\xi^2+\xi=1$ these factors are zero exactly when $a^2-a=\frac1{16}$. We now have a finite list of values for $a$, and we plug each of them together with $\xi$ into the equations to find all values of $b_1$ and $b_2$ for which our equations vanish for $p=2,3,4,5$. So, we get the list above. For each row we compute the corresponding value of $C$ from $c_1=\gamma$, i.e.
\[
C \= c_1(b_1,b_2,\xi,a)-\frac1{24}\sum_i \frac{1+Q_i}{1-Q_i} \=c_1(b_1,b_2,\xi,a)-\frac{2\xi+1}{12}\,.
\]

What remains is to prove that the $F_{A,B,C}$ satisfy the identities given in the last column. For the case $a=1/2$, this is easy, since $F_{A,B,C}$ splits as the product of two rank 1 cases, for which an identity is given in Theorem~\ref{r1}. For the case $a=3/4$, the identities follow directly by applying Theorem~\ref{tr} below, with $m=2$ and $A=1/2$, and again using identities from Theorem~\ref{r1}.

Only the case $a=1$ is a bit more work: using
\be
 (-xq^{1/2};q)_\infty = \sum_{k\ge 0} \frac{q^{\frac{1}{2}k^2} x^k}{(q)_k} \label{quant}
\ee
 (this is a direct consequence of  (7) in Chapter 2 of \cite{Zagier}), with $x=q^{-n/2}$, we find
\[ \bal
\sum_{m,n\ge 0} \frac{q^{\frac 12 m^2-\frac 12 mn +\frac 12 n^2}}{(q)_m (q)_n} &= \sum_{n\ge 0} \frac{q^{\frac 12 n^2} (-q^{-\frac 12 n+\frac 12})_\infty}{(q)_n} \\
&= \sum_{n\ge 0} \frac{q^{2 n^2} (-q^{-n+\frac 12})_\infty}{(q)_{2n}} + \sum_{n\ge 0} \frac{q^{2 n^2+2n+\frac 12} (-q^{-n})_\infty}{(q)_{2n+1}}.
\eal
\] 
Now using that for $n\ge 0$ we have $(-q^{-n+\frac 12})_\infty = q^{-\frac 12 n^2} (-q^{\frac 12})_n (-q^{\frac 12} )_\infty$ and $(-q^{-n})_\infty =2 q^{-\frac 12 n^2-\frac 12 n} (-q)_n (-q )_\infty$, this equals
\be \label{sla} \bal
&(-q^{\frac 12} )_\infty \sum_{n\ge 0} \frac{q^{\frac 32 n^2} (-q^{\frac 12})_n}{(q)_{2n}} + 2q^{\frac 12}(-q )_\infty \sum_{n\ge 0} \frac{q^{\frac 32 n^2+\frac 32 n} (-q)_n}{(q)_{2n+1}}\\
=&(-q^{\frac 12} )_\infty \sum_{n\ge 0} \frac{q^{\frac 32 n^2}}{ (q^{\frac 12})_n (q^2;q^2)_{n}} + 2q^{\frac 12}(-q )_\infty \sum_{n\ge 0} \frac{q^{\frac 32 n^2+\frac 32 n}}{(q)_n (q;q^2)_{n+1}}.
\eal \ee
To get identities for these last two sums, we use equations (19) and (44) in~\cite{Slater}, which (in our notation) read
\[ \bal
\sum_{n\ge 0} \frac{(-1)^n q^{3n^2}}{(-q;q^2)_n (q^4;q^4)_n } &=\frac{(q^2;q^5)_\infty (q^3;q^5)_\infty (q^5;q^5)_\infty }{(q^2;q^2)_\infty},\\
\sum_{n\ge 0} \frac{q^{\frac 32 n^2+\frac 32 n}}{(q)_n (q;q^2)_{n+1}} &= \frac{(q^2;q^{10})_\infty (q^8;q^{10})_\infty (q^{10};q^{10})_\infty }{(q)_\infty}.
\eal \]
If we use the Jacobi triple product identity $ (-xq^{1/2})_\infty (-x^{-1} q^{1/2})_\infty (q)_\infty =\sum_{n\in \Z} x^n q^{n^2/2}$ on the right hand sides and replace $q$ by $-q^{1/2}$ in the first identity, we get
\[ \bal
\sum_{n\ge 0} \frac{q^{\frac 32 n^2}}{ (q^{\frac 12})_n (q^2;q^2)_{n}} &= q^{\frac{7}{120}} \frac{\theta_{5,\frac 34} (2z)  +\theta_{5,\frac{13}{4}} (2z)  }{\eta(z)},\\
\sum_{n\ge 0} \frac{q^{\frac 32 n^2+\frac 32 n}}{(q)_n (q;q^2)_{n+1}} &= q^{-\frac{49}{120}} \frac{\theta_{5,2} (2z)}{\eta(z)}.
\eal \]
Further we have 
\[ \bal
(-q^{\frac 12})_\infty &= \frac{(q;q^2)_\infty}{(q^{\frac 12};q)_\infty} = \frac{(q)^2_\infty}{(q^2;q^2)_\infty (q^{\frac{1}{2}} ; q^{\frac{1}{2}})_\infty} = q^{\frac{1}{48}} \frac{\eta(z)^2}{\eta(2z)\eta(z/2)},\\ 
(-q)_\infty&= \frac{(q^2;q^2)_\infty}{(q)_\infty} = q^{-\frac{1}{24}} \frac{\eta(2z)}{\eta(z)},
\eal \]
and so we get from~\eqref{sla}
\[ F_{A,B,C} (q)= \frac{\eta(z)}{\eta(2z)\eta(z/2)} \left(\theta_{5,\frac 34} (2z)  +\theta_{5,\frac{13}{4}} (2z)\right) +2 \frac{\eta(2z)}{\eta(z)^2}\theta_{5,2} (2z),\]
where $A=\left(\begin{smallmatrix} 1&-1/2\\-1/2&1 \end{smallmatrix}\right)$,  $B=\left(\begin{smallmatrix}0\\0\end{smallmatrix}\right)$ and $C=-1/20$.

The proof for the identity for $B=\left(\begin{smallmatrix}-1/2\\0\end{smallmatrix}\right)$ and $C=1/20$ is very similar, and so we omit some of the details. We have
\[ \bal
\sum_{m,n\ge 0}&\frac{q^{\frac 12 m^2-\frac 12 mn +\frac 12 n^2-\frac 12 m}}{(q)_m (q)_n} \\
&= 2(-q )_\infty \sum_{n\ge 0} \frac{q^{\frac 32 n^2-\frac 12 n} (-q)_n}{(q)_{2n}} + (-q^{\frac 12} )_\infty \sum_{n\ge 0} \frac{q^{\frac 32 n^2+ n} (-q^{\frac 12})_{n+1}}{(q)_{2n+1}}\\
&=2(-q )_\infty \sum_{n\ge 0} \frac{q^{\frac 32 n^2-\frac 12 n}}{(q)_{n}(q;q^2)_n} + (-q^{\frac 12} )_\infty \sum_{n\ge 0} \frac{q^{\frac 32 n^2+ n} (-q^{\frac 12})_{n+1}}{(q)_{2n+1}}.
\eal \]
Again we use two identities from Slater's list (see~\cite{Slater}), namely (46) which reads
\[\sum_{n\ge 0} \frac{q^{\frac 32 n^2-\frac 12 n}}{(q)_{n}(q;q^2)_n} = \frac{(q^4;q^{10})_\infty (q^6;q^{10})_\infty (q^{10};q^{10})_\infty }{(q)_\infty},\]
and so we can identify it as $q^{-1/120} \theta_{5,1}(2z)/\eta(z)$, and (97), which should read (note that there are mistakes in some of the exponents; we have given the corrected version here)
\[ \bal
\sum_{n\ge 0}&\frac{q^{3 n^2+2 n} (-q;q^2)_{n+1}}{(q^2;q^2)_{2n+1}}\\
&= \frac{(-q;q^2)_\infty}{(q^2;q^2)_\infty} \left( (-q^{11};q^{30})_\infty (-q^{19};q^{30})_\infty -q^3  (-q;q^{30})_\infty (-q^{29};q^{30})_\infty\right)(q^{30};q^{30})_\infty\\
&=  \frac{(-q;q^2)_\infty}{(q^2;q^2)_\infty} (q^3;q^{10})_\infty (q^7;q^{10})_\infty (q^{10};q^{10})_\infty (q^4;q^{20})_\infty (q^{16};q^{20})_\infty.
\eal \]
If we replace $q$ by $q^{1/2}$, we find
\[  \sum_{n\ge 0} \frac{q^{\frac 32 n^2+ n} (-q^{\frac 12})_{n+1}}{(q)_{2n+1}} =q^{-\frac{17}{240}} \frac{\theta_{5,\frac 32} (z) \theta_{5,2} (2z) \eta(z)}{\eta (z/2) \eta(2z) \eta(10z)},\]
which gives the desired result.
\end{proof}

In~\cite{Zagier} one can find a list of triples $(A,B,C)$ for $r=2$ (Table~2 on p.\ 47) for which numerical experiments show that the condition~(iii) of Corollary~\ref{modcor} holds, as well as~(ii). We see that the cases of Theorem~\ref{afam} with $a=1$ are on this list, but the ones with $a=3/4$ appear to be new. We will come back to the case $a=3/4$ in the next section. 

Similar analysis for the other two families in~\eqref{fam} gave the following results. In both cases if the matrix in the family is diagonal then the modular forms are products of the ones from Theorem~\ref{r1}. Non-diagonal cases are listed in Tables~\ref{table:a2} and~\ref{table:a1}.

\begin{table}[ht]
\caption{ A complete list of modular triples $(A,B,C)$ with the matrix} $A=\begin{pmatrix}a&2-a\\2-a&a\end{pmatrix}$, $a>1$, $a\ne 2$. 
\centering
\begin{tabular}{llr|c}
&&&\\
A&B&C&$F_{A,B,C}(e^{2\pi i z})$\\
\hline
&&&\\
$\begin{pmatrix}\frac43&\frac23\\\frac23&\frac43\end{pmatrix}$&$\begin{pmatrix}0\\0\end{pmatrix}$&$-\frac1{30}$&? $\bal\frac1{\eta(z)}&\underset{n \in \Z}\sum (-1)^n\bigl(2q^{\frac{15}2(n+\frac3{10})^2}\\&+q^{\frac{15}2(n+\frac1{30})^2}-q^{\frac{15}2(n+\frac{11}{30})^2}\bigr)\eal$\\
&&&\\
&$\begin{pmatrix}-\frac23\\-\frac13\end{pmatrix}$ and $\begin{pmatrix}-\frac13\\-\frac23\end{pmatrix}$&$\frac1{30}$&? $\bal\frac1{\eta(z)}&\underset{n \in \Z}\sum (-1)^n\bigl(2q^{\frac{15}2(n+\frac1{10})^2}\\&+q^{\frac{15}2(n+\frac{13}{30})^2}-q^{\frac{15}2(n+\frac{23}{30})^2}\bigr)\eal$\\
&&&\\
\hline
&&&\\
$\begin{pmatrix}\frac32&\frac12\\\frac12&\frac32\end{pmatrix}$&$\begin{pmatrix}\frac14\\-\frac14\end{pmatrix}$ and $\begin{pmatrix}-\frac14\\\frac14\end{pmatrix}$&$-\frac1{120}$&$\theta_{5,1}(\frac z2)/\eta(\frac z2)$\\
&&&\\
&$\begin{pmatrix}\frac14\\\frac34\end{pmatrix}$ and $\begin{pmatrix}\frac34\\\frac14\end{pmatrix}$&$\frac{11}{120}$&$\theta_{5,2}(\frac z2)/\eta(\frac z2)$\\
\end{tabular}
\label{table:a2}
\end{table}
In Table~\ref{table:a2}, the identities for the case $a=3/2$ follow directly by applying Theorem~\ref{tr} with $m=2$ and $A=2$, and using identities from Theorem~\ref{r1}. For the case $a=4/3$ we were unable to find a proof, but we verified them to a high order in the power series in $q$.

\begin{table}[ht]
\caption{The list containing all $(B,C)$ such that $F_{A,B,C}$ is modular, where} $A=\begin{pmatrix}a&1-a\\1-a&a\end{pmatrix}$, $a>\frac12$, $a \ne 1$.
\centering
\begin{tabular}{ll|c}
&&\\
B&C&$F_{A,B,C}(e^{2\pi i z})$\\
\hline
&&\\
$\begin{pmatrix}b\\-b\end{pmatrix}$&$\dfrac{b^2}{2a}-\dfrac1{24}$&$\frac1{\eta(z)}\underset{n\in\Z+\frac ba}\sum q^{an^2/2}$\\
&&\\
$\begin{pmatrix}-\frac12\\-\frac12\end{pmatrix}$&$\dfrac1{8a}-\dfrac1{24}$&$\frac2{\eta(z)}\underset{n\in\Z+\frac 1{2a}}\sum q^{an^2/2}$\\
&&\\
$\begin{pmatrix}1-\frac a2\\\frac a2\end{pmatrix}$ and $\begin{pmatrix}\frac a2\\1-\frac a2\end{pmatrix}$ &$\dfrac a8-\dfrac1{24}$&$\frac1{2\eta(z)}\underset{n\in\Z+\frac 12}\sum q^{an^2/2}$
\end{tabular}
\label{table:a1}
\end{table}
In Table~\ref{table:a1}, the identity for $B=\left(\begin{smallmatrix}b\\-b\end{smallmatrix}\right)$ is given in~\cite{Zagier} (see (26) in Chapter 2). The proof uses that for any $n\in\Z$
\be \underset{k-l=n}{\sum_{k,l\ge 0}} \frac{q^{kl}}{(q)_k (q)_l} =\frac1{(q)_\infty}.\label{jac}\ee
The identity for $B=\left(\begin{smallmatrix}-\frac12\\-\frac12\end{smallmatrix}\right)$ is proven similarly, using 
\[ \underset{k-l=n}{\sum_{k,l\ge 0}} \frac{q^{kl-\frac12k -\frac12l}}{(q)_k (q)_l} =\frac{q^{n/2}+q^{-n/2}}{(q)_\infty},\]
for all $n\in\Z$. This identity follows directly from~\eqref{jac}:
\[ \bal \underset{k-l=n}{\sum_{k,l\ge 0}} \frac{q^{kl-\frac12k -\frac12l}}{(q)_k (q)_l} &=  \underset{k-l=n}{\sum_{k,l\ge 0}} \frac{q^{kl-\frac12k -\frac12l}\left((1-q^k)+q^k\right)}{(q)_k (q)_l}
\\
&=  \underset{k-l=n}{\sum_{k\ge 1,l\ge 0}} \frac{q^{kl-\frac12k -\frac12l}}{(q)_{k-1} (q)_l} +  \underset{k-l=n}{\sum_{k,l\ge 0}} \frac{q^{kl+\frac12k -\frac12l}}{(q)_k (q)_l}.\eal
\]
If we replace $k$ by $k+1$ in the first sum on the RHS, we see that it equals $q^{-n/2}/(q)_\infty$ and the second sum equals $q^{n/2}/(q)_\infty$.

To get the identity for $B=\left(\begin{smallmatrix}\frac a2\\1-\frac a2\end{smallmatrix}\right)$ we use
\be \underset{k-l=n}{\sum_{k,l\ge 0}} \frac{q^{kl+l}}{(q)_k (q)_l} = \frac1{(q)_\infty} (-1)^n q^{-\frac12n^2-\frac12n} s_n,\label{sums}\ee
with $s_n = \sum_{k\ge n} (-1)^k q^{\frac12k^2+\frac12k} $ (this is easily obtained by checking that both sides satisfy the recursion $b_n+q^{n+1}b_{n+1} =1/(q)_\infty$ and $\lim_{n\rightarrow \infty} b_n =\frac{1}{(q)_\infty}$) to get
\[ F_{A,B,C} (q)= \frac{q^{a/8}}{\eta(z)} \sum_{n\in\Z} q^{(a-1)(n^2+n)/2} s_n.\]
If we replace $n$ by $-n-1$ in the sum and use that $s_{-n-1}=s_{n+1}=s_n -(-1)^n q^{\frac 12 n^2+\frac 12 n}$, we easily get that $\sum_{n\in\Z} q^{(a-1)(n^2+n)/2} s_n=\frac 12 \sum_{n\in \Z} q^{a(n^2+n)/2}$, which gives the desired result.
 
We also checked for each matrix $A$ in Zagier's list for $r=2$ (p.\ 47 in~\cite{Zagier}) if the corresponding list of vectors $B$ is complete. It appears to be complete in all cases except $A=\begin{pmatrix}a&1-a\\1-a&a\end{pmatrix}$. For such matrices only the modular forms in the first row of Table~\ref{table:a1} were known.

\section{Counterexamples to Nahm's conjecture}

The Bloch group $B(K)$ of a field $K$ is an abelian group defined as the quotient of the kernel of the map
\be\lb{d}\bal
 \Z[K^*\setminus 1] \;&\to\; \Lambda^2 K^* \\
x &\mapsto x \wedge (1-x)
\eal\ee
by the subgroup generated by all elements of the form
\be\lb{rel}
 [x]\+[1-x]\,,\quad[x]\+\bigl[\frac 1x\bigr]\,,\quad[x]+[y]+[1-xy]+\bigl[\frac{1-x}{1-xy}\bigr]+\bigl[\frac{1-y}{1-xy}\bigr]\,.
\ee
If $K$ is a number field than $B(K)\otimes_{\Z} \Q \cong K_3(K) \otimes_{\Z} \Q$ and the regulator map is given explicitly on $B(K)$ by
\[\bal
 B(K) &\;\to\; \R^{r_2}\\	
 x &\mapsto (D(\sigma_1(x)), \dots, D(\sigma_{r_2}(x)))
\eal\]
where $r_2$ is the number of pairs of complex conjugate embeddings of $K$ into $\Co$, $\sigma_1,\dots,\sigma_{r_2}$ is any choice of such embeddings from different pairs, and
\[
 D(x) \= \Im\bigl( Li_2(x) + \log(1-x)\log|x| \bigr)
\]
is the Bloch-Wigner dilogarithm function. It vanishes on all combinations in~\eqref{rel}.

Let $(Q_1,\dots,Q_r)$ be an arbitrary solution of the system of algebraic equations~\eqref{ne} in some number field $K$. Then the element $[Q_1]+\dots+[Q_r] \in \Z[K^*\setminus 1]$ belongs to the kernel of~\eqref{d}. Indeed, we have 
\[
 \sum_i Q_i \wedge (1-Q_i) \=  \sum_i Q_i \wedge \prod_j Q_j^{A_{ij}} \= \sum_{i,j} A_{ij} Q_i \wedge Q_j \= 0
\]
because of the symmetry $A_{ij}=A_{ji}$. Hence every solution of~\eqref{ne} defines an element in the Bloch group of the corresponding field.

Recall that there exists the unique solution $(Q_1^0,\dots,Q_r^0)$ of~\eqref{ne} with $Q^0_i \in (0,1)$, and we have used this solution to compute the asymptotics of~\eqref{F} when $q \to 1$. If~\eqref{F} is a modular function then for this solution we have
\be\lb{rat}
 L(Q^0_1)\+\dots\+L(Q^0_r) \in \pi^2\Q
\ee
where $L(x)$ is the Rogers dilogarithm function (condition~(ii) of Corollary~\ref{modcor}). Rogers dilogarithm is defined in $\R$ and takes values in $\pi^2\Q$ on all combinations of real arguments of the form~\eqref{rel}. On the other hand, these are essentially all known functional equations for $L(x)$. Therefore it is very naturally to expect that $[Q^0_1]+\dots+[Q^0_r]$ is torsion in the corresponding Bloch group because of~\eqref{rat}. (It is automatically torsion if the field $\Q(Q^0_1,\dots,Q^0_r)$ is totally real.) Similar reasoning lead Werner Nahm to the following conjecture.

\begin{conjecture} 
For a positive definite symmetric $r\times r$ matrix with rational coefficients $A$ the following are equivalent:
\begin{itemize}
 \item[(i)] The element $[Q_1]+\dots+[Q_r]$ is torsion in the corresponding Bloch group for every solution of~\eqref{ne}.
 \item[(ii)] There exist $B\in\Q^r$ and $C\in\Q$ such that $F_{A,B,C}$ is a modular function.
\end{itemize}
\end{conjecture}

This conjecture is true in case $r=1$, and there are a lot of examples supporting the Conjecture also for $r>1$ (see \cite{Zagier}). Although examples show that it is not sufficient to require only $[Q^0_1]+\dots+[Q^0_r]$ to be torsion, it doesn't actually follow from anywhere that one should consider all solutions of~\eqref{ne} in~(i). We will see soon that this requirement is indeed too strong. 

As an example, let us consider matrices of the form $A=\begin{pmatrix}a&1-a\\1-a&a\end{pmatrix}$. The corresponding equations are
\[\bcs
1-Q_1 \= Q_1^a \; Q_2^{1-a}\,,\\
1-Q_2 \= Q_1^{1-a} Q_2^a\,,\\
\ecs\]
hence 
\[\bal
&\frac{1-Q_1}{Q_2}\=\Bigl(\frac{Q_1}{Q_2}\Bigr)^a\=\frac{Q_1}{1-Q_2}\,,\\
&(1-Q_1)(1-Q_2)\= Q_1 Q_2\,,\\
&Q_1 + Q_2 \= 1 \quad\Rightarrow\quad [Q_1]+[Q_2]\= 0 \text{ in } B(\Co)\,.
\eal\]
This computation is the same for all values of $a$ and we see from Table~\ref{table:a1} that indeed we have modular functions for every $a$. 

Next, let us look at the table from Theorem~\ref{afam}. One can check that the matrix $A=\begin{pmatrix}1&-1/2\\-1/2&1\end{pmatrix}$ satisfies condition~(i) of the Conjecture. (All solutions of~\eqref{ne} are $(Q_1,Q_2)=(x,x)$ with $1-x=x^{1/2}$.) However, $A=\begin{pmatrix}3/4&-1/4\\-1/4&3/4\end{pmatrix}$ does not satisfy~(i), and so we get a counterexample to Nahm's conjecture, since there do exist corresponding modular functions. Indeed, consider the corresponding equation:
\be\lb{ce}\bcs
1-Q_1 \= Q_1^{3/4} \; Q_2^{-1/4}\,,\\
1-Q_2 \= Q_1^{-1/4} Q_2^{3/4}\,.\\
\ecs\ee
It is algebraic equation in the variables $Q_1^{1/4}$ and $Q_2^{1/4}$. Let $t=Q_1^{1/4}Q_2^{-1/4}$. Then we have from the above equations
\[\bal
 &\frac{1-Q_1}{Q_2^{1/2}}\=t^3 \quad&\Rightarrow\quad Q_2^{1/2}\=t^{-3}(1-Q_1)\,,\\
 &\frac{1-Q_2}{Q_1^{1/2}}\=t^{-3} \quad&\Rightarrow\quad Q_1^{1/2}\=t^3(1-Q_2)\,,\\
\eal\]
and we substitute these equalities into $Q_1^{1/2}=t^2 Q_2^{1/2}$ to get
\[\bal
 &t^3(1-Q_2)\=t^2 t^{-3}(1-Q_1)\,,\\
 &t^4(1-Q_2)\=1\-Q_1\=1\-t^4 Q_2\,,\\
 &t^4\=1\,.\\
\eal\]
Consequently, all solutions of~\eqref{ce} are $(Q_1,Q_2)\=(x,x)$ where $x$ is a solution of
$1 \- x \= t x^{1/2}$ for a 4th root of unity $t^4=1$. Equivalently,
\[
 (1-x)^4\=x^2 \quad\Leftrightarrow\quad (x^2 - 3x + 1)(x^2 - x + 1)\= 0\,.
\]
We see that $(Q_1,Q_2)=\bigl(\frac{1+\sqrt{-3}}2,\frac{1+\sqrt{-3}}2\bigr)$ is a solution of~\eqref{ce}, and the corresponding element $2\bigl[\frac{1+\sqrt{-3}}2\bigr]$ is not torsion because $D\bigl(\frac{1+\sqrt{-3}}2\bigr) = 1.01494...$. Here $D$ is the Bloch-Wigner dilog (see \cite[Chapter I, Section 3]{Zagier}) for which it is known that $D(x)=0$ if and only if $x\in\R$.

A similar thing happens in Table~\ref{table:a2}: the matrix $A=\begin{pmatrix}4/3&2/3\\2/3&4/3\end{pmatrix}$ satisfies the Conjecture while  $A=\begin{pmatrix}3/2&1/2\\1/2&3/2\end{pmatrix}$ is a counterexample. So far we have two counterexamples, and we notice that both matrices match into the following general pattern. 

\begin{theorem}\label{tr} Let $A$ be a real positive definite symmetric $r \times r$-matrix, $B$ a vector of length $r$, and $C$ a scalar. For an arbitrary $m\ge1$ we define
\[
 A'\= {\rm I}_{mr} +{\rm E}_m \otimes (A-{\rm I}_r)\,,\quad B'\=l_{mr} +e_m \otimes (B-l_r)\,, \quad C'\=C/m,
\]
where $E_m\in M_{m\times m} (\Q)$ such that $(E_m)_{ij}= 1/m$, $e_m\in \Q^m$ such that $(e_m)_i=1/m$ and $l_r\in\Q^r$ such that $(l_r)_i=\frac{2i-r-1}{2r}$.
Then
\[
 F_{A',B',C'}(q)\=F_{A,B,C}(q^{1/m})\,.
\]
\end{theorem}
\begin{proof}
The proof relies on the following identity
\[ \frac{q^{\frac{1}{2} n^2}}{(q)_n} =\underset{k_1+\ldots +k_m=n}{\sum_{k \in (\Z_{\ge0})^m}} \frac{q^{\frac{m}{2}k^T k +m l_m^T k}}{(q^m;q^m)_{k_1} \cdots (q^m;q^m)_{k_m}} \]
which holds for all $n\ge 0$. It follows directly if we use~\eqref{quant} on both sides in the trivial identity
\[ (-xq^{1/2};q)_\infty = (-xq^{1/2};q^m)_\infty (-xq^{3/2};q^m)_\infty \cdots (-xq^{m-1/2};q^m)_\infty,\]
and compare the coefficient of $x^n$ on both sides.

Using the identity we find
\begin{equation*}
\bal
F_{A,B,C}(q) &= \sum_{n\in (\Z_{\ge0})^r} \frac{q^{\frac 12 n^T An + n^T B + C} }{(q)_{n_1}\dots (q)_{n_r}}\\
&= \sum_{n\in (\Z_{\ge0})^r} q^{\frac 12 n^T (A-{\rm I}_r)n + n^T B + C} \underset{mK e_m=n}{\sum_{K\in M_{r\times m}(\Z_{\ge0})}} \frac{q^{\frac{m}{2} ||K||^2 +mr e_r^TK l_m}}{(q^m;q^m)_K},
\eal
\end{equation*}
where $||K||^2 =\sum_{i=1}^r \sum_{j=1}^m K_{ij}^2$ and $(q;q)_K =\prod_{i=1}^r \prod_{j=1}^m (q;q)_{K_{ij}}$. Now changing the order of summation we get that this equals
\[ \sum_{K\in M_{r\times m}(\Z_{\ge0})}\frac{ q^{\frac{m^2}{2}e_m^TK^T(A-{\rm I}_r)Ke_m+ \frac{m}{2} ||K||^2 +me_m^T K^T B+mr e_r^TK l_m+C}}{(q^m;q^m)_K}.
\]
If we turn the $r\times m$ matrix $K$ into a vector of length $rm$ by putting the columns of $K$ under each other, we can recognize this last sum as $F_{A',B'',C'} (q^m)$, where $A'$ and $C'$ are as in the theorem and $B''= e_m \otimes B +rl_m \otimes e_r$. We can easily verify that
\[ rl_m \otimes e_r = l_{mr} -e_m\otimes l_r,\]
which gives $B''=B'$, with $B'$ as in the theorem. So we have found
\[ F_{A,B,C} (q) = F_{A',B',C'} (q^m).\]
Now replacing $q$ by $q^{1/m}$ gives the desired result.
\end{proof}
 
Let us take $r=1$ and $m=2$. Then
\[\bal
&A\=\frac12 \quad\rightsquigarrow\quad A'\=\begin{pmatrix}3/4&-1/4\\-1/4&3/4\end{pmatrix}\\
&A\=2 \quad\rightsquigarrow\quad A'\=\begin{pmatrix}3/2&1/2\\1/2&3/2\end{pmatrix}\\
\eal\]
and the theorem produces modular functions for these $2\times2$ matrices from the ones known for $r=1$. One can construct more counterexamples with higher $r$ using Theorem~\ref{tr}.

Finally, we would like to give one more counterexample, this time such that $A$ has integer entries. Let 
\[
A\=\begin{pmatrix}3&1&1&0\\1&3&0&1\\1&0&1&0\\0&1&0&1\end{pmatrix}\,,\quad B\=\frac12\begin{pmatrix}1\\-1\\1\\1\end{pmatrix}\,,\quad C\=\frac1{15}\,.
\]
All solutions of~\eqref{ne} in this case are 
\[
(Q_1,Q_2,Q_3,Q_4) \=\Bigl (u,u,\frac1{1+u},\frac1{1+u}\Bigr) \;\text{ with}\; 1-u^2=u^4
\]
and
\[
(Q_1,Q_2,Q_3,Q_4) \=\Bigl (u,-u,\frac1{1+u},\frac1{1-u}\Bigr) \;\text{ with}\; 1-u^2=-u^4\,.
\]
It is easy to check that solutions of the first type give torsion elements in the Bloch group, while ones of the second type give non-torsion elements. 
On the other hand, we have that
\[
F_{A,B,C}(q)=\frac{\eta(2z)^2\theta_{5,1}(z)}{\eta(z)^3}.
\]
We get this identity by applying the theorem below to $A=\left(\begin{smallmatrix} 3/2&1/2\\1/2&3/2\end{smallmatrix}\right)$, $B= \left(\begin{smallmatrix}1/4\\ -1/4\end{smallmatrix}\right)$ and $C=-1/120$, and using the identity for this case given in Table~\ref{table:a2}.
\begin{theorem}
Let $A$ be a real positive definite symmetric $r \times r$-matrix, $B$ a vector of length $r$, and $C$ a scalar. Let $A'$, $B'$ and $C'$ be the symmetric $2r \times 2r$-matrix, the vector of length $2r$ and the scalar, resp.\ , given by
\[
A' = \begin{pmatrix} 2A& {\rm I}_r\\ {\rm I}_r & {\rm I}_r\end{pmatrix}\,, \quad B' = \begin{pmatrix} 2B\\ \frac{1}{2} \\ \vdots \\ \frac{1}{2}\end{pmatrix}\,, \quad C'=2C+\frac{r}{24},
\]
then
\[ F_{A',B',C'} (q)= \frac{\eta(2z)^r}{\eta(z)^r} F_{A,B,C}(q^2).\]
\end{theorem}
\begin{proof}
Using $(q^2;q^2)_n =(q;q)_n (-q;q)_n$, $(q^2;q^2)_\infty =(q;q)_\infty (-q;q)_\infty$ and \eqref{quant}, we see that
\[ 
\frac{(q^2;q^2)_\infty}{(q;q)_\infty} \frac{1}{(q^2;q^2)_n} = \frac{(-q;q)_\infty}{(q;q)_n (-q;q)_n} = \frac{(-q^{n+1};q)_\infty}{(q;q)_n} = \frac{1}{(q)_n} \sum_{k\ge 0} \frac{q^{\frac{1}{2}k^2+\frac{1}{2}k +nk}}{(q)_k},
\]
and so 
\begin{equation*}
\bal
\frac{(q^2;q^2)_\infty^r}{(q)_\infty^r}& F_{A,B,C}(q^2) \\
& = \sum_{n\in (\Z \ge 0)^r} \frac{q^{n^TAn+2n^TB+2C}}{(q)_{n_1} \cdots (q)_{n_r}}\sum_{k\in (\Z \ge 0)^r} \frac{q^{\frac{1}{2}k^Tk +n^Tk +\frac{1}{2}(k_1+k_2+\ldots +k_r)}}{(q)_{k_1}\cdots (q)_{k_r}}\\
&= \sum_{n,k\in (\Z \ge 0)^r} \frac{q^{n^TAn +\frac{1}{2}k^Tk +n^Tk +2n^T B +\frac{1}{2}(k_1+k_2+\ldots +k_r)+2C}}{(q)_{n_1} \cdots (q)_{n_r}(q)_{k_1} \cdots (q)_{k_r}}.
\eal
\end{equation*}
If we turn the two vectors $n$ and $k$ into one vector of length $2r$ by putting $k$ below $n$, we can recognize this last sum as $q^{-r/24} F_{A',B',C'} (q)$, where $A'$, $B'$ and $C'$ are as in the theorem. So we have found
\[ \frac{(q^2;q^2)_\infty^r}{(q)_\infty^r} F_{A,B,C}(q^2) = q^{-r/24} F_{A',B',C'} (q).\]
Multiplying both sides by $q^{r/24}$ gives the desired result.
\end{proof}

\end{document}